\documentclass[11pt]{amsart}
\usepackage{amsmath,amssymb,amsfonts,amsthm, amscd,indentfirst}
\usepackage{amsmath,latexsym,amssymb,amsmath,
amscd,amsthm,amsxtra}
\usepackage{hyperref}\usepackage{url}
\usepackage{color}
\usepackage{pdfpages}
\usepackage{graphics}
\usepackage{enumitem}

\newtheorem{theorem}{Theorem}[section]
\newtheorem{prop}{Proposition}[section]

\def\rr{\mathbb{R}}
\def\ss{\mathbb{S}}

\def\bb{\mathbb{B}}

\def\O{\Omega}
\def\p{\partial}

\def\a{\alpha}

\def\l{\lambda}

\def\p{\partial}

\def\S{{\Sigma}}
\def\<{\langle}
\def\>{\rangle}
\def\div{{\rm div}}
\def\n{\nabla}
\def\G{\Gamma}

\def\De{\Delta}

\def\ep{\epsilon}

\numberwithin{equation} {section}

\begin{document}

\title[A partially overdetermined problem in a half ball]{A partially overdetermined problem in a half ball}
\author{Jinyu Guo}
\address{School of Mathematical Sciences\\
Xiamen University\\
361005, Xiamen, P.R. China}
\email{guojinyu14@163.com}
\author{Chao Xia}
\address{School of Mathematical Sciences\\
Xiamen University\\
361005, Xiamen, P.R. China}
\email{chaoxia@xmu.edu.cn}
\thanks{This work is supported by  NSFC (Grant No. 11501480, 11871406) and  the Natural Science Foundation of Fujian Province of China (Grant No. 2017J06003) and the Fundamental Research Funds for the Central Universities (Grant No. 20720180009).
}

\begin{abstract}
In this paper, we study a partially overdetermined mixed boundary value problem in a half ball. We prove that a domain on which this partially overdetermined problem admits a solution if and only if the domain is a spherical cap intersecting $\ss^{n-1}$ orthogonally. As an application, we show that a stationary point for a partially torsional rigidity under a  volume constraint must be a spherical cap.
\end{abstract}

\date{}
\maketitle

\medskip

\tableofcontents

\section{Introduction}
In a celebrated paper \cite{Se}, Serrin initiated the study of the following overdetermined boundary value problem (BVP)
\begin{equation}\label{21}
\begin{cases}{}
\Delta u=1, & \text{in}\ \Omega\\
u=0 ,     & \text{on} \ \partial\Omega\\
\p_\nu u=c, & \text{on}\ \partial\Omega,\\
\end{cases}
\end{equation}
where $\Omega$ is an open, connected, bounded domain in $\mathbb{R}^{n}$ with smooth boundary $\p\O$,   $c\in \rr$ is a constant and $\nu$ is the unit outward normal to $\partial\Omega$.
Serrin proved that if \eqref{21} admits a solution, then  $\Omega$ must be a ball and the solution $u$ is radially symmetric.
Serrin's proof is based on the moving plane method or Alexandrov reflection method, which
has been invented by Alexandrov in order to  prove the famous nowadays so-called Alexandrov's soap bubble theorem \cite{Al}: any closed, embedded hypersurface of constant mean curvature (CMC) must be a round sphere.

It is a common belief that the domain in which Serrin's overdetermined BVP admits a solution has close relationship with the CMC surfaces. On one hand, Serrin's and Alexandrov's proofs of their theorems share the same Alexandrov reflection method. On the other hand, Weinberger \cite{We} and Reilly \cite{Re} offered alternative proofs of Serrin's and Alexandrov's theorems respectively, based on an integral method. In particular, Reilly's  \cite{Re} proof utilizes the Dirichlet BVP
\begin{equation}\label{22}
\begin{cases}{}
\Delta u=1, & \text{in}\ \Omega\\
u=0 ,     & \text{on} \ \partial\Omega\\
\end{cases}
\end{equation}
and conclude that if $\p\O$ is of CMC, then the solution $u$ to \eqref{22} must satisfy $\p_\nu u=c$, that is, $u$ is a solution to \eqref{21}.
We refer to a nice survey paper by Nitsch-Trombetti \cite{NT} for an overview of Serrin type overdetermined problem and a collection of proofs to Serrin's theorem, as well as another nice survey paper by Magnanini \cite{Magnanini} on the relationship of Alexandrov's theorem and Serrin's theorem. See also \cite{Mag} and \cite{Mag2} for some stability results for Alexandrov's theorem and Serrin's theorem. We also refer to \cite{DPW} on a recent progress on the relationship of these two problems for unbounded domains.

{ Serrin's solution to the overdetermined BVP \eqref{21} is also closely related to an isoperimetric-type inequality for the  so-called torsional rigidity. For a bounded connected domain $\O$ with smooth boundary $\p\O$, the torsional rigidity is defined by
$$\tau(\O)=\sup_{0\neq u\in W_0^{1,2}(\O)} \frac{\left(\int_\O v dx\right)^2}{\int_\O |\n v|^2 dx}.$$ By the direct method of Calculus of Variations, the supremum in $\tau(\O)$ is achieved by a multiple of the function $u$
 satisfying \eqref{22}.
The isoperimetric problem for $\tau(\O)$ is answered by Saint-Venant's principle (see e.g. \cite{PO}), which states that  $\tau(\O)\le \tau(B_r)$, where $B_r$ is a ball such that  ${\rm Vol}(\O)={\rm Vol}(B_r)$, with equality holding if and only if $\O=B_r$.
On the other hand, the first variational formula or Hadamard's formula for $\tau(\O)$ tells that a stationary domain for $\tau(\O)$ among domains with fixed volume must satisfy $\p_\nu u=c$. Thus Serrin's solution to the overdetermined BVP \eqref{21} implies the only stationary domains for $\tau(\O)$ are balls.}

In this paper, we { propose to }study a partially overdetermined BVP where the domain lies in a ball.
Let $$\bb^n=\{x=(x_1,\cdots, x_n)\in \rr^n: |x|<1\}$$ be the open unit ball and $$\bb^n_+=\{x\in \bb^n: x_n>0\}$$ be the open unit half ball and $\ss^{n-1}=\p \bb^n$ be the unit sphere.
Let $\O\subset \bb^n$ be an open bounded, connected domain
whose boundary $\p\O$ consists of two parts $\bar{\Sigma}$ and $T=\p\O\setminus \bar{\Sigma}$, where $\S\subset \bb^n$ is a smooth hypersurface  and $T\subset \ss^{n-1}$ meets $\S$ at a common $(n-2)$-dimensional submanifold $\Gamma\subset \ss^{n-1}$.
In particular $\S$ is a smooth embedded hypersurface in $\bb^n$ with boundary  $\Gamma\subset\ss^{n-1}$.

{The free boundary CMC (or minimal) hypersurfaces in $\bb^n$, which means the hypersurfaces have constant mean curvature(or zero mean curvature)      intersecting with $\partial\bb^{n}$ orthogonally}, have recently attracted many attentions. Inspiring results are contained in a series of papers of Fraser-Schoen \cite{FS1, FS2}
 about minimal hypersurfaces with free boundary in a  ball and the first Steklov eigenvalue.
 The simplest examples of the CMC and minimal hypersurfaces  in $\bb^n$ with free boundary are the spherical caps and the geodesic disk in $\bb^n$ intersecting with $\ss^{n-1}$ orthogonally.
Ros-Souam \cite{RS}  have proved, among other results, an Alexandrov-type theorem in this setting by using the Alexandrov reflection method, see also \cite{Wente} for a similar problem in a half space. Precisely, they proved that if $\S$ is a smooth embedded hypersurface in $\bb^n$ intersecting with $\ss^{n-1}$ orthogonally, and $\Gamma\subset \ss^{n-1}_+$, the hemi-sphere, then $\S$ must be  a spherical cap or the geodesic disk.
Recently, Wang and the second author \cite{WX} found an alternative proof of an Alexandrov-type theorem in this setting when $\O\subset \bb^n_+$, based on an integral method and a solution to some mixed BVP, which is in the spirit of Reilly \cite{Re} and Ros \cite{Ro}.

{Similar to the torsional rigidity for bounded domains in Euclidean space, the following partial torsional rigidity for domains in $\bb_+^n$ can be defined:
\begin{eqnarray}\label{part_torsion}
\tilde{\tau}(\O)=\sup_{0\neq v\in W_0^{1,2}(\O, \S)} \frac{\left(\int_\O v dx\right)^2}{\int_\O |\n v|^2 dx- \int_T v^2 dA}.
\end{eqnarray}
That $\tilde{\tau}(\O)$ is well-defined can be seen in Section 2. One can show that $\int_\O |\n v|^2 dx- \int_T v^2 dA$ is coercive in $W_0^{1,2}(\O, \S)$, see \eqref{coercive} below. Thus by the direct method of Calculus of Variations, the supremum in $\tilde{\tau}(\O)$ is achieved by some multiple of  the unique solution $u$ of
\begin{equation}\label{bvp}
\begin{cases}{}
\Delta u=1 , & \text{in} \ \Omega,\\
u=0 ,     & \text{on} \ \bar \Sigma,\\
 \p_{\bar N}u=u , & \text{on}\ T,
\end{cases}
\end{equation}
where  $\bar N(x)=x$ is the outward unit normal of $T$.
Moreover, we have
\begin{eqnarray}\label{part_torsion1}
\tilde{\tau}(\O)=\int_\O |\n u|^2 dx- \int_T u^2 dA=- \int_\O u dx.
\end{eqnarray}

One may ask whether the isoperimetric-type problem for $\tilde{\tau}(\O)$, that is in fact a Saint-Venant-type principle, holds for $\tilde{\tau}(\O)$.
We remark that in this setting, a relative isoperimetric problem, which asks for the domain with least area among all domains in $\bb_+^n$ with preassigned volume, has its solution as the spherical caps and the geodesic disk in $\bb_+^n$ intersecting with $\ss^{n-1}$ orthogonally,  see \cite{BZ}, Theorem 18.1.3 or \cite{BS}. 

By computing the first variational formula for $\tilde{\tau}(\O)$, one can show (see Proposition \ref{maxim}) that a stationary domain for the partial torsional rigidity among domains with fixed volume carries a function which satisfies the following partially overdetemined BVP
\begin{equation}\label{overd}
\begin{cases}{}
\Delta u=1 , & \text{in} \ \Omega,\\
u=0 ,     & \text{on} \ \bar \Sigma,\\
 \p_{\nu}u=c, & \text{on}\ \bar \Sigma,\\
 \p_{\bar N}u=u , & \text{on}\ T.
\end{cases}
\end{equation}
The first result of this paper is  the following Serrin type rigidity result for \eqref{overd}.
\begin{theorem}\label{mainthm}
 Let $\O$ be an open bounded, connected domain in $\bb^n_+$
whose boundary $\p\O=\bar \S\cup T$, where $\S\subset\bar \bb^{n}_+$ is a smooth hypersurface and $T\subset \ss^{n-1}$  meets $\S$ at a common $(n-2)$-dimensional submanifold $\Gamma\subset \ss^{n-1}$. Let $c\in \rr.$
Assume the partially overdetermined BVP \eqref{overd} admits a weak solution $u\in W_0^{1,2}(\O, \S)$. Assume further that
\begin{eqnarray}\label{reg-ass}
u\in W^{1,\infty}(\O)\cap W^{2,2}(\O).
\end{eqnarray}
 Then $c>0$ and, for some $a\in \ss^{n-1}$, $\O$ must be equal to the domain
 \begin{eqnarray}\label{sph_cap}
\O_{nc}(a):=\{x\in \bb^n_+: |x-a\sqrt{1+(nc)^2}|^2< (nc)^2\}, \quad a\in \ss^{n-1} \end{eqnarray}
 and \begin{eqnarray}\label{cap-sol}
u(x)=u_{a,nc}(x):=\frac{1}{2n}(|x-a\sqrt{1+(nc)^2}|^2- (nc)^2).
\end{eqnarray}
\end{theorem}
We remark that one part of  $\p \O_{nc}(a)$, given by $$C_{nc}(a):=\{x\in \bb^n_+: |x-a\sqrt{1+(nc)^2}|^2=(nc)^2\}$$ is a spherical cap with radius $nc$ which intersects $\ss^{n-1}$ orthogonally.}

By a weak solution  to \eqref{overd}, we mean $$u\in W_0^{1,2}(\O, \S):=\{u\in W^{1,2}(\O), u|_{\bar \S}=0\}$$ satisfies
\begin{eqnarray}\label{weak-form0}
&& \int_\O (\<\n u, \n v\> +v )\, dx -\int_T uv \, dA=0, \hbox{ for all } v\in W_0^{1,2}(\O, \S)
\end{eqnarray}
together with an additional boundary condition $\p_\nu u=c$ on $\S$. A regularity result by Lieberman \cite{Liebm} shows that a weak solution $u$ to \eqref{overd} belongs to $C^\infty(\bar \O\setminus \G)\cap C^\a(\bar \O)$ for some $\a\in (0,1)$. Higher order regularity up to the interface $\G=\bar \S\cap \bar T$  is a subtle issue for mixed boundary value problems.
The regularity assumption \eqref{reg-ass} is for technical reasons, that is, we will use an integration method which requires \eqref{reg-ass} to perform integration by parts.

 We remark that we do not assume that $\S$ meets $\ss^{n-1}$ orthogonally a priori. It seems hard to use the Alexandrov reflection method as Ros-Souam \cite{RS}.
On the other hand, because of the lack of regularity of $u$ on $\G$, it is difficult to use the maximum principle as Weinberger's \cite{We}.

Nevertheless, if $\S$ meets $\ss^{n-1}$ orthogonally,  we can prove the regularity \eqref{reg-ass} (see Proposition \ref{regularity}). Therefore, we have the following
\begin{theorem}\label{mainthm1.5}
 Let $\O$ be as in Theorem \ref{mainthm}. Assume in addition that $\S$ meets $T$
orthogonally at  $\Gamma\subset \ss^{n-1}$. Let $c\in \rr.$
Assume the partially overdetermined BVP \eqref{overd} admits a weak solution $u\in W_0^{1,2}(\O, \S)$.
 Then $c>0$ and, for some $a\in \ss^{n-1}$, $\O$ must be equal to the domain $\O_{nc}(a)$ in \eqref{sph_cap} and $u$ must be equal to the function $u_{a,nc}$ defined in \eqref{cap-sol}.
\end{theorem}

We use a purely integral method to prove Theorem \ref{mainthm}. Unlike the usual integration, our integration makes use of a weight function $x_n$, which is non-negative by the assumption $\O\subset\bb^n_+$.  A crucial ingredient in this paper is a conformal Killing vector field $X_n$, which has been proved to be a powerful tool in \cite{WX}. See the end of Section 2 for its definition and properties. By using $X_n$, we get a Pohozaev-type identity, Proposition \ref{Poho}. Then with the usual $P$-function $P=|\n u|^2-\frac{2}{n}u$, we show the identity $$\int_\O x_n u\De Pdx=0.$$ Theorem \ref{mainthm} follows since the $P$-function is subharmonic.


A closely related overdetermined problem in a hemi-sphere has been considered by Qiu and the second author \cite{QX} (see also \cite{CV}).
We also call attention to a similar partially overdetermined problem in a convex cone which has been considered recently by Pacella-Tralli \cite{PA} (see also \cite{CR} for its generalization to general elliptic operators). Their partially overdetermined BVP  is of mixed Dirichlet-Neumann type. Compared to their BVP, one of the difficulties in our case is that the classical maximum principle can not be used directly. Another novelty in the ball case is the use of weight integration.

\


{ Back to the isoperimetric type problem for the partial torsional rigidity, as we already mentioned, our solution to the partially overdetermined BVP \eqref{overd}  in $\O\subset\bb^n_+$ can be used to characterize the stationary domain for the partial torsional rigidity.}
\begin{theorem}\label{mainthm2} Let $\O$ be a stationary domain under fixed volume for the partial torsional rigidity $\tilde{\tau}$. Assume that the function $u$ which attains the supremum in $\tilde{\tau}(\O)$ lies in $W^{1,\infty}(\O)\cap W^{2,2}(\O)$. Then $\O$ must equal the domain $\O_r(a)$ given in \eqref{sph_cap}.
\end{theorem}

Here $\O$ is called stationary for $\tilde{\tau}$ under a fixed volume if for any variation $\O^t\subset \bb^n_+$ with fixed volume,
$$\frac{d}{dt}\Big|_{t=0}\tilde{\tau}(\O^t)=0.$$
We show that if $\O$ is stationary, then the solution $u$ which attains the supremum in $\tilde{\tau}(\O)$ satisfies $\p_\nu u=c$ on $\S$ in addition to \eqref{bvp}, see Section \ref{part-torsion}. That is, $\O$ admits a solution to the mixed overdetermined BVP \eqref{overd}.
Therefore, Theorem \ref{mainthm} implies Theorem \ref{mainthm2}.

{For the isoperimetric-type problem, if one can prove that the supremum $$\sup_{\O\subset\bb^n_+, |\O|={\rm const.}}\tilde{\tau}(\O)$$ is attained by an open domain $\O_0$ with boundary part $\p \O_0\cap \bb^n_+$ being of class $C^2$ and intersecting $\ss^n$ orthogonally, then one may use  Theorem \ref{mainthm2} to get the isoperimetric-type inequality for the partial torsional rigidity.

However, the existence for a maximizer for the shape optimization problem is a subtle problem, even in the classical case. A good reference on the shape oprimization is the book \cite{Henrot}. In general, the maximizer for this kind of shape optimization problems cannot be attained if one restricts to the family of open sets because of the loss of compactness. On the other hand,  the maximizer can be proved to exist if one extends the family of open sets to the so-called quasi-open sets. The problem now reduces to the question that whether a quasi-open maximizer is open as well as it is regular. We hope to address the existence problem in the future.
}
\

The rest of the paper is organized as follows. In Section 2, we study two kinds of eigenvalue problems in $\O$ and use them to prove the existence and uniqueness of the mixed BVP \eqref{bvp}. We also review the conformal Killing vector field $X_n$ and its properties. In Section 3, we prove a weighted Pohozaev inequality and then Theorem \ref{mainthm}. In Section 4, we study the partial torsional rigidity and prove Theorem \ref{mainthm2}.

\

\section{Mixed boundary value problem}

From this section on,  Let $\O$ be an open bounded, connected domain in $\bb^n_+$
whose boundary $\p\O=\bar \S\cup T$, where $\S\subset\bar \bb^{n}_+$ is a smooth hypersurface and $T\subset \ss^{n-1}$  meets $\S$ at a common $(n-2)$-dimensional submanifold $\Gamma\subset \ss^{n-1}$.

 We consider  the following two kinds of eigenvalue problems in $\O$.\\
\noindent{\bf I. Mixed Robin-Dirichlet eigenvalue problem}
 \begin{equation}\label{mixRD_eigen}
\begin{cases}{}
\Delta u=-\l u, &
\hbox{ in } \Omega,\\
u= 0,&\hbox{ on }\bar \S,\\
 \p_{\bar N}u=  u, &\hbox{ on } T.
\end{cases}
\end{equation}
The first Robin-Dirichlet eigenvalue can be variationally characterized  by
\begin{eqnarray}\label{lamet}
\l_1=\inf_{0\ne u\in W_0^{1,2}(\O, \S)}\frac{\int_\O |\n u|^2 dx-\int_T u^2 dA}{\int_\O u^2 dx}.
\end{eqnarray}

\noindent{\bf II. Mixed Steklov-Dirichlet eigenvalue problem.} (see e.g. \cite{Agr, BKPS})
 \begin{equation}\label{mixSD_eigen}
\begin{cases}{}
\Delta u=0, &
\hbox{ in } \Omega,\\
u= 0,&\hbox{ on } \bar \S,\\
 \p_{\bar N}u= \mu u, &\hbox{ on } T.
\end{cases}
\end{equation}
The mixed Steklov-Dirichlet eigenvalue can be seen as the eigenvalue of the Dirichlet-to-Neumann map
\begin{eqnarray*}
\mathcal{L}: &L^2(T)\to &L^2(T)\\
&u\mapsto & \p_{\bar N}\hat u
\end{eqnarray*}
where $\hat u\in W_0^{1,2}(\O, \S)$ is the harmonic extension of $u$ to $\O$ satisfying $u=0$ on $\S$. According to the spectral theory for compact, symmetric linear  operators, $\mathcal{L}$ has a discrete spectrum $\{\mu_i\}_{i=1}^\infty$ (see e.g. \cite{Agr} or \cite{ BKPS}),
$$0<\mu_1\le \mu_2 \le\cdots \to +\infty.$$
The first eigenvalue $\mu_1$ can be variationally characterized  by
\begin{eqnarray}\label{mue}
\mu_1=\inf_{0\ne u\in W_0^{1,2}(\O, \S)}\frac{\int_\O |\n u|^2 dx}{\int_T u^2 dA}.
\end{eqnarray}

In our case, we have
\begin{prop}\label{mixSD_eigen1}\
\begin{itemize}\item[(i)] $\l_1(\O)\ge 0$ and $\l_1=0$ iff $\O= \bb^n_+$.
\item[(ii)] $\mu_1(\O)\ge 1$ and $\mu_1=1$ iff $\O= \bb^n_+$.
\end{itemize}
\end{prop}

\begin{proof}
If $\O=\bb^n_+$, one checks that $u=x_n\ge 0$ indeed solves \eqref{mixRD_eigen} with $\l=0$ and \eqref{mixSD_eigen} with $\mu=1$. Since $u=x_n$ is a non-negative solution, it must be the first eigenfunction and hence $\l_1(\bb^n_+)=0$ and $\mu_1(\bb^n_+)=1.$

{On the other hand, for $\O \subset \bb^n_+$, assume $u\in W^{1,2}_{0}(\Omega;\Sigma)$ and $\bar{u}$ is standard zero extension of $u$, then $\bar{u}\in W^{1,2}_{0}(\bb^{n}_{+};\partial\bb^{n}_{+}\backslash\partial\bb)$, therefore applying the variational characterization \eqref{lamet} and \eqref{mue}, one sees that  $\l_1(\O)\ge \l_1(\bb^n_+)=0$ and $\mu_1(\O)\ge\mu_1(\bb^n_+)=1.$}

If $\O\subsetneqq \bb^n_+$, then the Aronszajn unique continuity theorem implies $\l_1(\O)>\l_1(\bb^n_+)=0$.  For $\mu_1$, it has been proved in  \cite{BKPS}, Proposition 3.1.1,  that  $\mu_1(\O)>\mu_1(\bb^n_+)=1$.
\end{proof}

Since \begin{eqnarray}\label{poincare}
\int_{\O} |\n u|^2 dx-\int_T u^2 dA\ge \l_1 \int_{\O} u^2 dx,\hbox{ for } u\in W^{1,2}_0(\O,\S),
\end{eqnarray}
it follows from Proposition \ref{mixSD_eigen1} (i) that the partial torsional rigidity $\tilde{\tau}$ in \eqref{part_torsion} is well-defined.

Using Proposition \ref{mixSD_eigen1} (ii), we show the existence and uniqueness of the mixed BVP.
\begin{prop}\label{existence}Let $f\in C^\infty(\bar \O)$ and $g\in C^\infty(\bar T)$. Then the mixed BVP
\begin{equation}\label{mixSD_nonh}
\begin{cases}{}
\Delta u=f, &
\hbox{ in } \Omega,\\
u= 0,&\hbox{ on } \bar \S,\\
 \p_{\bar N}u= u+g, &\hbox{ on } T.
\end{cases}
\end{equation}
admits a unique solution $u\in C^\infty(\bar \O\setminus \G)\cap C^\a(\bar \O)$ for some $\a\in (0,1)$.
\end{prop}

\begin{proof} The weak solution to \eqref{mixSD_nonh} is defined to be $u\in W_0^{1,2}(\O, \S)$ such that
\begin{eqnarray}\label{weak-form}
&&B[u,v]:= \int_\O \<\n u, \n v\> \, dx -\int_T uv \, dA=\int_\O -fv\, dx+\int_T gv \, dA \hbox{ for all } v\in W_0^{1,2}(\O, \S).
\end{eqnarray}
From Proposition \ref{mixSD_eigen1}, we know $1-\frac{1}{\mu_1}>0$. Since
$$\int_T u^2 \, dA\le \frac{1}{\mu_1}\int_\O |\n u|^2 \, dx, \hbox{ for }u\in W_0^{1,2}(\O, \S),$$
we see
\begin{eqnarray}\label{coercive}
B[u, u]\ge (1-\frac{1}{\mu_1}) \int_\O |\n u|^2\, dx\ge  \frac{\l_1(\mu_1-1)}{\mu_1(1+\l_1)}\|u\|^{2}_{W_{0}^{1,2}(\O, \S)}
\end{eqnarray}

Thus $B[u,v]$ is coercive on $W_0^{1,2}(\O, \S)$.
The standard Lax-Milgram‘s theorem holds for the weak formulation to \eqref{mixSD_nonh}. Therefore,  \eqref{mixSD_nonh} admits a unique weak solution $u\in W_0^{1,2}(\O, \S)$.

The regularity $u\in  C^\infty(\bar \O\setminus \G)$ follows from the classical regularity theory for elliptic equations and $u\in C^\a(\bar \O)$ has been proved by Lieberman \cite{Liebm}, Theorem 2. Note that the global wedge condition in Theorem 2 in \cite{Liebm} is satisfied for the domain $\O$ whose boundary parts $\S$ and $T$ meet at a common in smooth $(n-2)$-dimensional manifold, see page 426 of  \cite{Liebm}.

\end{proof}

\begin{prop}\label{mp}Let $u$ be the unique solution to \eqref{mixSD_nonh} with $f\ge 0$ and $g\le 0$.
Then either $u\equiv0$ in $\O$ or $u<0$ in $\O\cup T$.
\end{prop}

\begin{proof} Since the Robin boundary condition has an unfavorable sign, we cannot use the maximum principle directly.
Since $u_{+}=\text{max}\{u , 0\}\in W_0^{1,2}(\O, \S)$, we can use it as a test function in the weak formulation \eqref{weak-form} to get
\begin{alignat*}{2}
\int_{\Omega}-f u_+ \, dx + \int_T g u_+ \, dA&=\int_{\Omega}|\n u_+|^2 dx-\int_T (u_+)^2\, dA.
\end{alignat*}
Since $f\ge 0$ and $g\le 0$, we have
\begin{alignat*}{2}
\int_{\Omega}-f u_+ \, dx + \int_T g u_+ \, dA\le 0.
\end{alignat*}
On the other hand, it follows from \eqref{poincare} that
\begin{alignat*}{2}
\int_{\Omega}|\n u_+|^2 dx-\int_T (u_+)^2\, dA\ge \l_1\int_\O (u_+)^2 dx\ge 0.
\end{alignat*}
From above, we conclude that $u_+\equiv 0$, which means $u\le 0$ in $\O$. Finally, by the strong maximum principle, we get either $u\equiv0$ in $\Omega$ or $u<0$ in $\Omega\cup T$.

\end{proof}

\begin{prop} Let $e^{T}$ be a tangent vector field to $T$. Let $u$ be the unique solution to \eqref{mixSD_nonh}. Then
\begin{equation}\label{bdry-prop}
 \langle(\n^{2}u)\bar N, e^T\rangle=0 \quad on\ T.
\end{equation}
where $\n^{2}u$ is the Hessian of $u$.
\end{prop}
\begin{proof}
By differentiating the equation $\p_{\bar N}u=u$ with respect to $e^{T}$, we get
\begin{alignat*}{2}
\n_{e^{T}}(u)&=\n_{e^{T}}(\<\n u,\bar N\>)= \langle(\n^{2}u)\bar N, e^T\rangle+\<\n u, \nabla_{e^{T}}\bar N\>\\
&=\langle(\n^{2}u)\bar N, e^T\rangle+\<\n u, e^T\>.
\end{alignat*}
Here we use the fact $\nabla_{e^{T}}\bar N=e^T$. The assertion \eqref{bdry-prop} follows.\end{proof}

\

In the rest of this section, we introduce an important conformal Killing vector field in $\bb^n$ (see \cite{WX}) which plays a crucial role in the proof of Theorem \ref{mainthm}.
Let
$$X_n:= x_n x-\frac12(|x|^2+1)E_n.$$ where $E_{n}=(0,\cdots,0,1)$ in $\mathbb{R}^{n}$.

One can check directly that the Lie derivative of $\delta_{ij}$ along $X_{n}$ satisfies
\begin{eqnarray}\label{XXeq1}
L_{X_n}\delta_{ij}:= \frac12(\p_i X_n^j+\p_j X_n^i)= x_n \delta_{ij}
\end{eqnarray}
and
\begin{eqnarray}\label{XXeq2}
X_n|_{\ss^{n-1}}= -E_n^T \hbox{ and }\<X_n, \bar N\>=0 \hbox{ on }\ss^{n-1}
\end{eqnarray}
 where $E_n^T= E_n-x_n x$ is the tangential projection of $E_n$ on $\ss^{n-1}$.

\

\section{Partially overdetermined BVP}

In this section we will use a method totally based on integral identities and inequalities to prove Theorem \ref{mainthm}. First we introduce $P$-function as follow
\begin{equation}\label{60}
P:=|\n u|^{2}-\frac{2}{n}u
\end{equation}

\begin{prop}\label{sub-harm}$\Delta P\ge 0 \hbox{ in }\O.$
\end{prop}
\begin{proof} By direct computation and using $\Delta u=1$,
\begin{equation*}
\Delta P(x)= 2|\n^2 u|^{2}+\<\n u, \n \De u\>-\frac{2}{n}\Delta u\\
\ge \frac 2 n(\De u)^2-\frac{2}{n}\Delta u=0.\end{equation*}
\end{proof}

Due to the lack of regularity, we need the following formula of integration by parts, see \cite{PA}, Lemma 2.1. (The original statement \cite{PA}, Lemma 2.1 is for a sector-like domain in a cone. Nevertheless, the proof is applicable in our case). We remark that a general version of integration-by-parts formula for Lipschitz domains has been stated in some classical book by Grisvard \cite{Grisv}, Theorem 1.5.3.1. However, it seems not enough for our purpose.

\begin{prop}[\cite{PA}, lemma 2.1]\label{int-by-parts}
Let $F:\O \to \rr^n$ be a vector field such that
\begin{eqnarray*}F\in C^1(\O\cup \S\cup T)\cap L^2(\O)\quad \hbox{ and } \quad \div(F)\in L^1(\O).
\end{eqnarray*}
Then \begin{eqnarray*}
\int_\O \div(F) dx=\int_{\S} \<F, \nu\> dA+ \int_{T} \<F, \bar N\> dA.
\end{eqnarray*}
\end{prop}

We first prove a Pohozaev-type identity for \eqref{overd}.
\begin{prop}\label{Poho}
Let  $u$ be the unique weak solution to \eqref{overd} such that $u\in W^{1,\infty}(\O)\cap W^{2,2}(\O)$. Then we have
\begin{equation}\label{Peq}
\int_{\Omega}x_n(P-c^{2})dx=0.
\end{equation}
\end{prop}

\begin{proof}
  First of all, we remark that, due to our assumption $u\in W^{1,\infty}(\O)\cap W^{2,2}(\O)$, Proposition \ref{int-by-parts} can be applied in all the following integration by parts.

Now we consider the following differential identity
\begin{eqnarray}\label{Peq1}
\div(uX_{n}-\langle X_{n},\nabla u\rangle\nabla u)&=&\langle X_{n}, \nabla u\rangle+u\div X_{n}-\langle(\nabla X_{n})\nabla u,\nabla u\rangle\\
&&\quad -\frac{1}{2}\langle X_{n},\nabla|\nabla u|^{2}\rangle-\langle X_{n}, \nabla u\rangle\Delta u\nonumber\\
&=&nx_{n}u-x_{n}|\nabla u|^{2}-\frac{1}{2}\langle X_{n},\nabla|\nabla u|^{2}\rangle.\nonumber
\end{eqnarray}
where we use equation $\De u=1$ and \eqref{XXeq1}.

Integrating by parts and using \eqref{XXeq2} and boundary conditions \eqref{overd}, we see that
\begin{equation}\label{Peq2}
-c^{2}\int_{\Sigma}\langle X_{n},\nu\rangle dA-\int_{T}\langle X_{n},\nabla u\rangle udA=\int_{\Omega}(nx_{n}u-x_{n}|\nabla u|^{2}+\frac{1}{2}|\nabla u|^{2}\div X_{n})dx-\frac{ c^{2}}{2}\int_{\Sigma}\langle X_{n},\nu\rangle dA.
\end{equation}
It follows that
\begin{equation}\label{Peq3}
\int_{\Omega}(nx_{n}u+(\frac{n}{2}-1)x_{n}|\nabla u|^{2})dx=-\frac{1}{2}c^{2}\int_{\Sigma}\langle X_{n},\nu\rangle dA-\int_{T}\langle X_{n},\nabla u\rangle udA
\end{equation}

Further integration by parts and using \eqref{XXeq1} yields
\begin{eqnarray}
&&\frac12c^2\int_\S\<X_n, \nu\>dA= \frac12c^2\left(\int_\O {\rm div}X_ndx -\int_T\<X_n, \bar N\>dA\right)= \frac n 2 c^2\int_\O x_ndx.\label{Peq5}
\end{eqnarray}
Since $u\in W^{2,2}(\O)$, we know by the trace theorem (see e.g. \cite{Grisv}) that $\n^T u\in W^{1,2}(T, \rr^n)$ where $\n^T u$ denotes the tangent component of $\n u$ on $T$. Doing integration by parts on $T$, using \eqref{XXeq2}, we have
\begin{eqnarray}
\int_T \<X_n, \n u\>udA&=&\int_T \<-E_n^T, \n(\frac12u^2)\>dA\nonumber
\\&=&-\int_{\G} \frac12u^2 \<E_n^T, \mu_{\G}\>ds+\int_T \frac12u^2{\rm div}_{T}E_n^TdA\nonumber
\\&=&\frac{1-n}{2}\int_T x_nu^2dA.\label{Peq6}
\end{eqnarray}
In the last equality we have used $u=0$ on $\G$ and the fact $${\rm div}_{T}E_n^T= (1-n)\<E_n, \nu_{T}\>=(1-n)x_n.$$

 To achieve \eqref{Peq}, we do a further integration by parts to get
 \begin{eqnarray*}
 \int_\O  x_n |\n u|^2dx
 &=&\int_\O (\div(x_n u \n u)-x_n u\De u-u\p_n u)dx\nonumber\\&=&\int_{T} x_n u\<\n u, \bar N\>dA + \int_{\S} x_n u\<\n u, \nu\>dA- \int_\O( x_n u\De u+u\p_n u)dx\nonumber
 \\&=& \int_T x_n u^2dA-\int_\O x_n udx -\int_T \frac12x_n u^2dA\nonumber
 \\&=& \int_T \frac12x_n u^2dA-\int_\O x_n udx.
\end{eqnarray*}
It follows that
 \begin{eqnarray}\label{Peq10}
\frac12 \int_T x_n u^2dA=  \int_\O ( x_n |\n u|^2+ x_n u)dx.
\end{eqnarray}
Substituting \eqref{Peq5}-\eqref{Peq10} into \eqref{Peq3}, we arrive at \eqref{Peq}.
\end{proof}

\begin{prop}\label{int-id}Let  $u$ be the unique solution to \eqref{overd} such that $u\in W^{1,\infty}(\O)\cap W^{2,2}(\O)$ and $P$ is defined by \eqref{60}. Then\begin{eqnarray}\label{Xeq0}
\int_\O x_n u\Delta P \, dx=0.
\end{eqnarray}
\end{prop}
\begin{proof} Since $u\in W^{1,\infty}(\O)\cap W^{2,2}(\O)$, we see $$\div(x_n u \n P- P\n (x_n u))\in L^1(\O) \quad \hbox{ and }\quad (x_n u \n P- P\n (x_n u))\in L^2(\O).$$

Firstly, we consider the following differential identity

\begin{eqnarray}
&&\div(x_{n}u\nabla P-P\nabla(x_{n}u))+c^{2}\div(x_{n}\nabla u-u\nabla x_{n})\nonumber\\
\qquad&=&x_{n}u\Delta P-P\Delta(x_{n}u)+c^{2}x_{n}-c^{2}\Delta x_{n}\nonumber\\
\qquad&=&x_{n}u\Delta P-2P\partial_{n}u-Px_{n}+c^{2}x_{n}.\label{Xeq1}
\end{eqnarray}
where we use the equation $\Delta u=1$ in $\Omega$.

Applying divergence theorem in \eqref{Xeq1} and  boundary conditions (1.6), we have

\begin{eqnarray}\label{Xeq2}
&&-c\int_{\Sigma}Px_{n}dA+\int_{T}(x_{n}u\partial_{\bar{N}}P-2x_{n}uP)dA+c^{3}\int_{\Sigma}x_{n}dA\\
&=&\int_{\Omega}(x_{n}u\Delta P-2\partial_{n}uP-Px_{n}+c^{2}x_{n})dx\nonumber\\
&=&\int_{\Omega}(x_{n}u\Delta P-2\partial_{n}uP)dx.\nonumber
\end{eqnarray}
where the last equation we use Proposition \ref{Poho}.

Noting that $P=c^2$ on $\S$. It follows from \eqref{Xeq2}
\begin{eqnarray}\label{Xeq3}
\int_{\Omega}x_{n}u\Delta Pdx=\int_{\Omega}2\partial_{n}uPdx+\int_{T}(x_{n}u\partial_{\bar{N}}P-2x_{n}uP)dA.
\end{eqnarray}

Using integration by parts, we have
\begin{eqnarray}
\int_\O 2\p_n u Pdx &= &\int_\O 2\p_n u (|\n u|^{2}-\frac{2}{n}u)dx\nonumber\\
&=&\int_T 2x_n u(|\n u|^{2}-\frac{2}{n}u)dA-\int_\O 2\p_n(|\n u|^{2}-\frac{2}{n}u)udx\nonumber
\\&=&\int_T 2x_n u(|\n u|^{2}-\frac{2}{n}u)dA-\int_\O 2  \left(\p_i (u^2)\p^2_{in} u-\frac{1}{n}\p_n (u^2)\right)dx\nonumber
\\&=&\int_T  \left(2x_n u|\n u|^2- 2u^2 \langle(\n ^2 u)\bar N, E_n\rangle- \frac{2}{n}x_n u^2 \right)dA+ \int_\O 2   u^2\p_n(\De u)dx\nonumber
\\&=&\int_T  \left(2x_n u|\n u|^2- 2x_n u^2 \langle(\n ^2 u)\bar N, \bar N\rangle- \frac{2}{n}x_n u^2 \right)dA\label{Xeq3.5}
.\end{eqnarray}
In the last equality we used \eqref{bdry-prop}.

Also
\begin{eqnarray}\label{Xeq4}
\int_{T}(x_n u\p_{\bar N} P- 2x_n u P)dA&=&\int_T \left(2x_n u\langle(\n^2 u)\n u, \bar N\rangle -\frac 2 n x_n u \p_{\bar N} u- 2 x_n u(|\n u|^{2}-\frac{2}{n}u)\right)dA\nonumber\\
&=&\int_T\left( 2x_n u^2\langle(\n^2 u)\bar N, \bar N\rangle +\frac 2 n x_n u^2- 2 x_n u|\n u|^{2}\right)dA
\end{eqnarray}
In the last equality we again used \eqref{bdry-prop} and also $\p_{\bar N} u=u$ on $T$.

Finally, substituting \eqref{Xeq3.5}-\eqref{Xeq4} into \eqref{Xeq3}, we get the conclusion \eqref{Xeq0}.
\end{proof}

\noindent{\bf Proof of Theorem \ref{mainthm}.}
From Propositions \ref{mp} and \ref{sub-harm} as well as $\O\subset \bb^n_+$, we have
\begin{eqnarray}\label{equiv}
x_n u\Delta P\le 0 \hbox{ in }\O.
\end{eqnarray}
We also know from Proposition \ref{int-id}
that $\int_\O x_n u\Delta P \, dx=0$.
It follows that $$x_n u\Delta P\equiv 0 \hbox{ in }\O.$$

Since $u<0$ in $\O$ by Proposition \ref{mp}, we conclude
$\De P\equiv 0$ in $\O$. From the proof of Proposition \ref{sub-harm}, we see immediately that
 $\n^{2}u$ is proportional to the identity matrix
in $\Omega$. Since $\Delta u=1$, we get
$u=\frac{1}{2n}|x-p|^2+A$ for some $p\in \rr^n$ and $A\in \rr$.

By the connectedness of $\Omega$ and $u=0$ on $\Sigma$, we conclude that $\S$ is part of a round sphere.
Using $\p_\nu u=c$ one verifies that $\O=\O_{nc}(a)$ and $u=u_{a, nc}$.
\qed


\

Theorem \ref{mainthm} has been proved under the regularity assumption \eqref{reg-ass} on $u$. In the special case when $\S$ meets $T$ orthogonally at $\G$, we can prove \eqref{reg-ass} is always true.
\begin{prop}\label{regularity}
Let $\O$ be as in Theorem \ref{mainthm}. Assume that $\S$ meets $T$ orthogonally. Let $u\in W_0^{1,2}(\O, \S)$ be a weak solution to \eqref{bvp}. Then $u\in W^{2,2}(\O)\cap C^{1,\a}(\bar \O)$ for some $\alpha\in(0,1)$.
\end{prop}
\begin{proof} Our proof follows the one in Pacella-Tralli \cite{PA} closely. Instead of flattening the boundary of the barrier and using planar reflection, here we use directly spherical reflection.

By the classical regularity theory for elliptic equations, we know that $u\in C^{\infty}(\bar{\Omega}\backslash\Gamma)$. It remains to prove the regularity up to $\Gamma$.

Denote $v(x):=u(x)e^{-|x|}$ for any $x\in \bar{\Omega}$. Then \eqref{bvp} becomes
\begin{equation}\label{trans}
\begin{cases}{}
Lv(x)=1 , & \text{in} \ \Omega,\\
v(x)=0 ,     & \text{on} \ \bar \Sigma,\\
 \p_{\bar N}v=0 , & \text{on}\ T,
\end{cases}
\end{equation}
where $Lv(x):=e^{|x|}\Delta v+2e^{|x|}\langle\frac{x}{|x|},\nabla v\rangle+e^{|x|}(1+\frac{n-1}{|x|})v(x)$.\\

Fix a point $x_{0}\in \Gamma$, we take a small neighborhood $U_{0}$ of $x_{0}$ in $\rr^n$ and consider the spherical reflection of $\tilde{\Omega}:=U_{0}\cap\Omega$:
\begin{eqnarray*}
 \Phi:  \tilde{\Omega}\to\tilde{\Omega}_{ref}:=\Phi(\tilde{\Omega})
\\ x\mapsto \frac{x}{|x|^{2}}.\qquad\qquad
\end{eqnarray*}
$\Phi$ is a diffeomorphism between $\tilde{\Omega}$ and $\tilde{\Omega}_{ref}$.
Note that the condition that $\S$ meets $T$ orthogonally guarantees domain $\Omega_{0}:=\widetilde{\Omega}\cup\widetilde{\Omega}_{ref}\cup (U_0\cap T)$
is $C^{2}$ near $x_{0}$. Denote $\tilde{\S}:=U_0\cap\S$ and $\tilde{\Sigma}_{ref}:=\Phi(\tilde{\S})$.  
\\ Next, we define
\begin{equation}\label{trans2}
w(x)=
\begin{cases}
v(x), &  {x\in \widetilde{\Omega}\cup (U_0\cap T)}\\
v\left(\frac{x}{|x|^{2}}\right), &  {x\in \widetilde{\Omega}_{ref}}.
\end{cases}
\end{equation}
Using the boundary condition of (\ref{trans}), we can check that $w(x)\in C^{2}(\Omega_{0})$.
In fact, we only need to check $w$ is $C^2$ across $\Omega_{0}\cap T$. 
By direct computation, we have, for $p\in \Omega_{0}\cap T$,
\begin{itemize}
  \item  [(1)] $\lim_{x\to p^-}w=\lim_{x\to p^+}w$;

  \item [(2)] $\p_{\bar N^-} w=\p_{\bar N^+} w=0$ at $p$;

  \item [(3)] $\p^2_{\bar N\bar N^-}w=\p^2_{\bar N\bar N^+}w=e^{-1}(\partial^2_{\bar N\bar N}u-u)$ at $p$.
\end{itemize}
Here $\p_{\bar N^{\mp}}$ means the left (right) first derivative  and
$$\p^2_{\bar N\bar N^{\mp}}w=\lim_{t\to 0^{\mp}}\frac{1}{t^2}\left(w(p+t\bar N)+w(p-t\bar N)-2 w(p)\right)$$
means the left (right) second derivative from $\tilde{\Omega}$ ($\widetilde{\Omega}_{ref}$) along $\bar N$.
The $C^2$-continuity of $w$ across $\O_0\cap T$ follows.

Moreover, in view of \eqref{trans} with \eqref{trans2}, we define a uniformly elliptic operator $Q$ in $\Omega_{0}$ as follows
\begin{equation*}
 Qw:=a_{ij}w_{ij}+b_kw_k+cw \ \text{in} \ \Omega_{0}
\end{equation*}
where
\begin{equation*}\label{A-ij}
a_{ij}(x)=
\begin{cases}
e^{|x|}\delta_{ij}, & {x\in \widetilde{\Omega}\cup (U_0\cap T)}\\
e^{\frac{1}{|x|}}|x|^{4}\delta_{ij}, & {x\in \widetilde{\Omega}_{ref}},
\end{cases}
\end{equation*}
\begin{equation*}\label{b-k}
\qquad\qquad b_{k}(x)=
\begin{cases}
2e^{|x|}\cdot\frac{x_{k}}{|x|}, &  {x\in \widetilde{\Omega}\cup (U_0\cap T)}\\
e^{\frac{1}{|x|}}((4-2n)|x|^{2}-2|x|)x_{k}, &  {x\in \widetilde{\Omega}_{ref}},
\end{cases}
\end{equation*}
\begin{equation*}\label{c}
\qquad c(x)=
\begin{cases}
e^{|x|}(1+\frac{n-1}{|x|}), &  {x\in \widetilde{\Omega}\cup (U_0\cap T)}\\
e^{\frac{1}{|x|}}(1+(n-1)|x|), &  {x\in \widetilde{\Omega}_{ref}}.
\end{cases}
\end{equation*}
Then we have $Qw =1$ in $\O_0$. We observe that $a_{ij}\in C^{0}(\bar{\Omega}_{0}),$ $b_{k}\in L^{\infty}(\Omega_{0})$ and $c\in C^{0}(\bar \Omega_{0})$.
Since $w(x)=0$ along $\bar{\tilde{\S}}\cap \tilde{\S}_{ref}$,
 we can deduce from  Theorem 9.15 in \cite{GN} that $w\in W^{2,p}(U)$ for some neighborhood $U\subset \bar \Omega_{0}$ of $x_0$ for any $p$.  Restricting $w$ to a neighborhood of $x_0$ in $\O$ and taking account that $u$ is smooth in the interior of $\O$, we conclude that $u\in W^{2,p}(\O)$ for any $p$. By Sobolev-Morrey's embedding theorem, $u\in C^{1,\alpha}(\bar \O)$. 
 \end{proof}
\noindent{\bf Proof of Theorem \ref{mainthm1.5}.} Theorem \ref{mainthm1.5} follows from Proposition \ref{regularity} and Theorem \ref{mainthm}. \qed

\
\section{Partial torsional rigidity}\label{part-torsion}

In this section, we study the partial torsional rigidity. We first derive the Hadamard variational formula.

\begin{prop}\label{Had}
Let $\S^t, t\in (-\ep, \ep)$ be a smooth variation of $\S$ given by a family of embedding $x^t: \ss^{n-1}_+\to \bb_+^n$ such that $x^0(\ss^{n-1}_+)=\S$ and $\frac{d}{dt}|_{t=0} x^t= Y\in T\ss^{n-1}$.  Let $\O^t$ be the enclosed domain by $\S^t$ and $\ss^{n-1}$.
Then
\begin{eqnarray*}
\frac{d}{dt}\Big|_{t=0}\tilde{\tau}(\O^t)=\int_{\S} (\p_\nu u)^2\<Y,\nu\> dA.
\end{eqnarray*}
\end{prop}

\begin{proof}
Let $u(t,\cdot)$ be the unique solution of \eqref{bvp} for $\O=\O^t$.
Then $$\tilde{\tau}(\O^t)=-\int_{\O^t} u(t,x) dx.$$
Denote $u= u(0,x)$ and $u'(0,\cdot)=\frac{\p}{\p t}\Big|_{t=0}u(t,\cdot)$.
Thus
\begin{eqnarray*}
\frac{d}{dt}\Big|_{t=0}\tilde{\tau}(\O^t)=-\int_{\S} u\<Y,\nu\> dA- \int_\O u'(0,x)dx= - \int_\O u'(0,x)dx.
\end{eqnarray*}
By taking derivative with respect to $t$ for
\begin{equation}\label{Heq1}
\begin{cases}{}
\Delta u(t, x)=1, & x\in \Omega,\\
u(t, x^t)=0,     & x^t\in\S^t,\\
\p_{\bar N}u(t, x) =u(t,x), &x \in T^t,
\end{cases}
\end{equation}
we get
\begin{equation}\label{Heq2}
\begin{cases}{}
\Delta u'(0, x)=0, & \hbox{ in }\Omega,\\
u'(0, x)+  \<\n u(0, x), Y\>=0,     &\hbox{ on }\S,\\
\p_{\bar N}u'(0, x) =u'(0,x), &\hbox{ on }T.
\end{cases}
\end{equation}
It follows from integration-by-parts in Proposition \ref{int-by-parts} and \eqref{Heq2} that
\begin{eqnarray*}
-\int_\O u'(0,x)dx &=& -\int_\O u'(0,x)\De u(0, x)dx\\&=&\int_{T}\partial_{\bar{N}}u'(0, x)u(0,x)-u'(0, x)\partial_{\bar{N}}u(0, x) dA
\\&&+ \int_{\Sigma}\partial_{\nu}u'(0, x)u(0,x)-u'(0, x)\partial_{\nu}u(0, x) dA
\\&=&\int_{T}u(0,x) u'(0,x) -u'(0,x)u(0, x) dA+ \int_\S \<\n u(0, x), Y\>\p_{\nu} u(0, x) dA
\\&=&\int_\S (\p_{\nu} u)^2\<Y,\nu\> dA.
\end{eqnarray*}
The assertion follows.
\end{proof}

\begin{prop}\label{maxim}
Let $\O$  be stationary for $\tilde{\tau}$ among all domains with fixed volume. Let $u$ be the unique weak solution of \eqref{bvp} in $\O$, Then $u$ satisfies in addition that $\p_{\nu} u$ is a constant along $\S$.
\end{prop}

\begin{proof} For the same variation as in Proposition \ref{Had}, it is well-known that
$$\frac{d}{dt}\Big|_{t=0}{\rm Vol}(\O^t)=\int_\S\<Y,\nu\>dA.$$
Using the Hadamard formula, since $\O_{\max}$ is a maximizer, we find there exists a Lagrangian multiplier $\l$ such that
$$\int_\S (\p_{\nu} u)^2\<Y,\nu\> dA=\l \int_\S\<Y,\nu\>dA.$$
This holds for all $Y\in T\ss^{n-1}$. Thus $\p_{\nu} u$ is a constant along $\S$.
\end{proof}

\noindent{\bf Proof of Theorem \ref{mainthm2}.} Theorem \ref{mainthm2} follows from Proposition \ref{maxim} and Theorem \ref{mainthm}. \qed

\

{\bf Acknowledgements.} We are indebted to Professor Guofang Wang for numerous insight and discussion on this topic. We also thank Professor Martin Man-Chun Li for his interest. We are grateful to the anonymous referee who attracts our attention to the existence problem and the excellent book \cite{Henrot} on the shape optimization, as well as his/her numerous suggestions which help to improve the paper considerably.

\

\end{document}